\documentclass[a4paper,11pt]{amsart}
\usepackage{amssymb}
\usepackage{mathptmx}

\title{The index of a threefold canonical singularity}

\author{Masayuki Kawakita}
\address{Research Institute for Mathematical Sciences, Kyoto University, Kyoto 606-8502, Japan}
\email{masayuki@kurims.kyoto-u.ac.jp}

\theoremstyle{plain}
\newtheorem{theorem}{Theorem}[section]
\newtheorem{proposition}[theorem]{Proposition}
\newtheorem{lemma}[theorem]{Lemma}
\newtheorem{corollary}[theorem]{Corollary}
\newtheorem{question}[theorem]{Question}
\newtheorem*{question'}{Question \ref{qsn:Shokurov}$\mathbf{'}$}

\theoremstyle{definition}
\newtheorem{definition}[theorem]{Definition}
\newtheorem{example}[theorem]{Example}

\theoremstyle{remark}
\newtheorem{remark}[theorem]{Remark}
\newtheorem*{acknowledgements}{Acknowledgements}

\newcommand{\bA}{\mathbb{A}}
\newcommand{\bN}{\mathbb{N}}
\newcommand{\bQ}{\mathbb{Q}}
\newcommand{\bZ}{\mathbb{Z}}
\newcommand{\cK}{\mathcal{K}}
\newcommand{\cO}{\mathcal{O}}
\newcommand{\fm}{\mathfrak{m}}

\DeclareMathOperator{\md}{md}
\DeclareMathOperator{\Supp}{Supp}
\DeclareMathOperator{\wt}{wt}

\begin{document}
\begin{abstract}
The index of a $3$-fold canonical singularity at a crepant centre is at most $6$.
\end{abstract}

\maketitle

\section{Introduction}
Let $P\in X$ be a normal $\bQ$-Gorenstein singularity. Shokurov asked if one can bound the index $r_P$ of $X$ at $P$ in terms of the discrepancies of divisors over $X$.

Suppose that $X$ has log canonical singularities with $P$ a log canonical centre. In $\dim X=2$, $r_P$ is $1$, $2$, $3$, $4$ or $6$ by the classification of singularities. In an arbitrary dimension, Ishii \cite{I00} and Fujino \cite{F01} reduced the boundedness of $r_P$ to a conjectural boundedness of a quotient of the birational automorphism group of a variety $S$ with $K_S\sim0$. In particular, they proved $r_P\le66$ in $\dim X=3$.

Suppose that $X$ has canonical singularities. In $\dim X=2$, $P$ is a rational double point, so $r_P=1$. The purpose of this paper is to provide an affirmative answer in $\dim X=3$. 

\begin{theorem}\label{thm:main}
Let $P\in X$ be a $3$-fold canonical singularity such that $P$ is a crepant centre. Then the index of $X$ at $P$ is at most $6$.
\end{theorem}

Here a crepant centre means the centre of a divisor with discrepancy zero. The condition that $P$ is a crepant centre is necessary even for a strictly canonical singularity, see Example \ref{exl:unbounded}. On the other hand, if once the minimal discrepancy at $P$ is fixed, then one can bound $r_P$ for an arbitrary $3$-fold canonical singularity $P\in X$ (Theorem \ref{thm:md}).

We shall prove Theorem \ref{thm:main} by using the singular Riemann--Roch formula (singRR) \cite{R87} due to Reid. In Sect.\ \ref{sec:crepant}, we build a tower $Y\to X$ of crepant blow-ups with $\bQ$-factorial terminal $Y$, on which the singRR is applicable unconditionally. Then we construct a divisor $F$ on $Y$ which possesses the information on the index $r_P$. The $r_P$ is determined by the Euler characteristics $\chi(iK_Y|_F)$, which can be explicitly computed by the singRR (Sect.\ \ref{sec:singRR}). We derive a numerical classification of the singularities on $Y$ together with $r_P$ in Sect.\ \ref{sec:boundedness}, by the method \cite{K01}, \cite{K05} in the classification of $3$-fold divisorial contractions. The boundedness of indices in terms of minimal discrepancies is discussed in Sect.\ \ref{sec:md}.

We work over an algebraically closed field $k$ of characteristic zero. A germ $P\in X$ means an algebraic germ of a variety $X$ at a closed point $P$.

\section{Crepant blow-ups}\label{sec:crepant}
Let $X$ be a normal $\bQ$-Gorenstein variety.

\begin{definition}
The \textit{index} of $X$ at a point $P$ is the smallest positive integer $r$ such that $rK_X$ is a Cartier divisor at $P$.
\end{definition}

Consider a normal variety $Y$ with a proper birational morphism $f\colon Y\to X$. A prime divisor $E$ on any such $Y$ is called a divisor \textit{over} $X$, and the image $f(E)$ is called the \textit{centre} of $E$ on $X$ and denoted by $c_X(E)$. The valuation $v_E$ on the function field of $X$ given by such $E$ is called an \textit{algebraic valuation} of $X$. If we write
\begin{align*}
K_Y=f^*K_X+\sum_Ea_E(X)E\quad\textrm{with $a_E(X)\in\bQ$},
\end{align*}
then $a_E(X)$ is called the \textit{discrepancy} of $E$. We say that $X$ has \textit{log canonical}, \textit{log terminal}, \textit{canonical}, \textit{terminal} singularities if $a_E(X)\ge-1$, $>-1$, $\ge0$, $>0$ respectively for all exceptional divisors $E$ over $X$.

The notion of crepancy is crucial in this paper.

\begin{definition}
\begin{enumerate}
\item
A \textit{crepant divisor} over $X$ is an exceptional divisor $E$ over $X$ with $a_E(X)=0$. A \textit{crepant valuation} of $X$ is the algebraic valuation $v_E$ given by a crepant divisor $E$.
\item
A \textit{crepant centre} on $X$ is the centre $c_X(E)$ of a crepant divisor $E$.
\item
A \textit{crepant blow-up} $f\colon Y\to X$ is a projective birational morphism from a normal variety $Y$ such that $K_Y=f^*K_X$.
\end{enumerate}
\end{definition}

\begin{remark}
\begin{enumerate}
\item
Suppose that $X$ is canonical. Then the number of crepant valuations of $X$ is finite. The complement of the union of all crepant centres is the largest terminal open subvariety of $X$.
\item
If $Y\to X$ is a crepant blow-up, then $X$ is canonical if and only if so is $Y$.
\end{enumerate}
\end{remark}

We have a crepant blow-up by the LMMP.

\begin{proposition}
Let $X$ be a variety with canonical singularities and $v$ a crepant valuation of $X$. Then there exists a crepant blow-up $f\colon Y\to X$ such that
\begin{enumerate}
\item
$Y$ is $\bQ$-factorial,
\item\label{itm:crepant_one}
$f$ has exactly one exceptional divisor $E$, and $v_E=v$,
\item
$-E$ is $f$-nef.
\end{enumerate}
\end{proposition}

\begin{proof}
Take a projective resolution of singularities $g\colon Z\to X$, and denote by $E_Z$ the divisor on $Z$ with $v_{E_Z}=v$. Take a Cartier divisor $H>0$ on $X$ whose support contains all the crepant centres. We write $g^*H=H_Z+F$ with the strict transform $H_Z$ of $H$, and $m$ for the coefficient of $E_Z$ in $F$. Fix $\epsilon>0$ so that $(Z,\epsilon(H_Z+2(F-mE_Z))$ is klt, and run $(K_Z+\epsilon(H_Z+2(F-mE_Z))$-LMMP over $X$ by \cite{BCHM10} to get a log minimal model $f\colon Y\to X$.

By $K_Z+\epsilon(H_Z+2(F-mE_Z))\equiv_XK_Z+\epsilon(F-2mE_Z)$, the negativity lemma \cite[Lemma 2.19]{K+92} shows that this LMMP contracts exactly all the $g$-exceptional divisors but $E_Z$, and $-E$ is $f$-nef for the strict transform $E$ of $E_Z$. Hence $f$ is a required crepant blow-up.
\end{proof}

\begin{remark}
If $X$ is $\bQ$-factorial, then (\ref{itm:crepant_one}) implies that $\rho(Y/X)=1$ and $-E$ is $f$-ample.
\end{remark}

\begin{corollary}\label{cor:crepant}
Let $X=X_0$ be a variety with canonical singularities and $Z$ a crepant centre on $X$. Then there exists a sequence of crepant blow-ups $f_t\colon X_t\to X_{t-1}$ for $1\le t\le s$ such that
\begin{enumerate}
\item
$X_t$ is $\bQ$-factorial for $t\ge1$ and $X_s$ is terminal,
\item
for $t\ge1$, $f_t$ has exactly one exceptional divisor $E_t$ and $-E_t$ is $f_t$-nef,
\item
$f_1(E_1)=Z$.
\end{enumerate}
\end{corollary}

We construct a divisor on $X_s$ which possesses the information on the index of $X$.

\begin{theorem}\label{thm:divisor}
Let $P\in X$ be a canonical singularity such that $P$ is a crepant centre. Let $r_P$ denote the index of $X$ at $P$ and $\fm_P$ the maximal ideal sheaf for $P$. Then there exist a crepant blow-up $f\colon Y\to X$ and an effective divisor $F$ on $Y$ supported in $f^{-1}(P)$ such that
\begin{enumerate}
\item
$Y$ is $\bQ$-factorial and terminal,
\item
for $i\in\bZ$,
\begin{align*}
f_*\cO_Y(iK_Y-F)&=
\begin{cases}
\fm_P\cO_X(iK_X)&\textrm{if $r_P\mid i$,}\\
\cO_X(iK_X)     &\textrm{otherwise,}
\end{cases}\\
R^jf_*\cO_Y(iK_Y-F)&=0\quad\textrm{for $j\ge1$}.
\end{align*}
\end{enumerate}
\end{theorem}

\begin{proof}
We take a sequence of crepant blow-ups $f_t$ in Corollary \ref{cor:crepant} with $Z=P$, and set $Y:=X_s$. We will construct inductively divisors $F_t\ge0$ on $X_t$ such that
\begin{align}
\label{eqn:F1R0}
f_{1*}\cO_{X_1}(iK_{X_1}-{F_1})&=
\begin{cases}
\fm_P\cO_X(iK_X)&\textrm{if $r_P\mid i$,}\\
\cO_X(iK_X)     &\textrm{otherwise,}
\end{cases}\\
\label{eqn:F1Rj}
R^jf_{1*}\cO_{X_1}(iK_{X_1}-{F_1})&=0\quad\textrm{for $j\ge1$},
\end{align}
and for $t>1$, 
\begin{align}
\label{eqn:FtR0}
f_{t*}\cO_{X_t}(iK_{X_t}-{F_t})&=\cO_{X_{t-1}}(iK_{X_{t-1}}-{F_{t-1}}),\\
\label{eqn:FtRj}
R^jf_{t*}\cO_{X_t}(iK_{X_t}-{F_t})&=0\quad\textrm{for $j\ge1$}.
\end{align}
Then Leray's spectral sequence induces that $F:=F_s$ is a required divisor.

We set $F_1:=E_1$. The vanishing (\ref{eqn:F1Rj}) follows from Kawamata--Viehweg vanishing theorem \cite[Theorem 1.2.5, Remark 1.2.6]{KMM87}. If $r_P\mid i$, then (\ref{eqn:F1R0}) is by the projection formula. To see (\ref{eqn:F1R0}) for $r_P\nmid i$, we regard $K_X$ as a fixed divisor (not a divisor class), and so $K_{X_1}=f_1^*K_X$. Denote by $\cK_X$ the constant sheaf of the function field of $X$. Then the inclusion $f_{1*}\cO_{X_1}(iK_{X_1}-{F_1})\subset\cO_X(iK_X)$ is interpreted by the expressions
\begin{align*}
f_{1*}\cO_{X_1}(iK_{X_1}-{F_1})&=\{u\in\cK_X\mid(u)_{X_1}+if_1^*K_X-F_1\ge0\},\\
\cO_X(iK_X)&=\{u\in\cK_X\mid(u)_X+iK_X\ge0\}.
\end{align*}
Suppose $u\in\cK_X$ satisfies $(u)_X+iK_X\ge0$. If $r_P\nmid i$, then $(u)_X+iK_X$ is not Cartier at $P$, so there exists a divisor $D>0$ passing through $P$ such that $(u)_X+iK_X-D$ is an effective Cartier divisor. Then $(u)_{X_1}+if_1^*K_X-f_1^*D\ge0$. By $f_1^*K_X=K_{X_1}$ and $F_1\subset\Supp f_1^*D$, we obtain $(u)_{X_1}+if_1^*K_X-F_1\ge0$, implying (\ref{eqn:F1R0}).

For $t>1$, we set $F_t:=\lceil f_t^*F_{t-1}\rceil$ inductively. $F_t=f_t^*F_{t-1}+c_tE_t$ with some $c_t\in[0,1)$, so $-F_t$ is $f_t$-nef. The (\ref{eqn:FtRj}) is again by Kawamata--Viehweg vanishing theorem. If $c_t=0$, then (\ref{eqn:FtR0}) is obvious. If $c_t>0$, then the equality $iK_{X_t}-F_t=f_t^*(iK_{X_{t-1}}-F_{t-1})-c_tE_t$ shows that $iK_{X_{t-1}}-F_{t-1}$ is not Cartier at every point in $f_t(E_t)$. Now we get (\ref{eqn:FtR0}) just as in the proof of (\ref{eqn:F1R0}) for $r_P\nmid i$.
\end{proof}

\section{The singular Riemann--Roch formula}\label{sec:singRR}
We shall apply the singular Riemann--Roch formula due to Reid to our crepant blow-up, and use the method \cite{K01}, \cite{K05} in the classification of $3$-fold divisorial contractions. We briefly recall the formula on a canonical $3$-fold.

\begin{theorem}[{\cite[Theorem 10.2]{R87}}]
Let $X$ be a projective $3$-fold with canonical singularities and $D$ a divisor on $X$ such that $D\sim i_PK_X$ with $i_P\in\bZ$ at each $P\in X$.
\begin{enumerate}
\item
There is a formula of the form
\begin{align*}
\chi(\cO_X(D))=\chi(\cO_X)+\frac{1}{12}D(D-K_X)(2D-K_X)+\frac{1}{12}D\cdot c_2(X)+\sum_Pc_P(D),
\end{align*}
where the summation takes place over the singularities of $\cO_X(D)$, and $c_P(D)\in \bQ$ is a contribution due to the singularity at $P$, depending only on the analytic type.
\item
For a terminal cyclic quotient singularity $P$ of type $\frac{1}{r_P}(1,-1,b_P)$,
\begin{align*}
c_P(D)=-\overline{i_P}\frac{r_P^2-1}{12r_P}+\sum_{j=1}^{\overline{i_P}-1}\frac{\overline{jb_P}(r_P-\overline{jb_P})}{2r_P},
\end{align*}
where $\overline{i}=i-\lfloor\frac{i}{r_P}\rfloor r_P$ denotes the residue of $i$ modulo $r_P$.
\item
For an arbitrary terminal singularity $P$,
\begin{align*}
c_P(D)=\sum_Qc_Q(D_Q),
\end{align*}
where $\{(Q,D_Q)\}_Q$ is a flat deformation of $(P, D)$ to the basket of terminal cyclic quotient singularities $Q$. Such $Q$ is called a fictitious singularity.
\end{enumerate}
\end{theorem}

\begin{remark}
The condition $D\sim i_PK_X$ always holds if $X$ is $\bQ$-factorial and terminal \cite[Corollary 5.2]{Km88}.
\end{remark}

Our object is a germ of a crepant blow-up $f\colon Y\to X$ with a divisor $F$ on $Y$ in Theorem \ref{thm:divisor} at a $3$-fold canonical singularity $P\in X$ with index $r_P$. Shrinking and compactifying it, we may assume that $Y$ is projective and terminal ($f$ is merely a projective morphism outside a neighbourhood of $P$). We shall express the function $\delta_P(i)$ below.

\begin{definition}
We define the function $\delta_P(i)$ on $\bZ$ as
\begin{align*}
\delta_P(i):=
\begin{cases}
1&\textrm{if $r_P\mid i$,}\\
0&\textrm{otherwise.}
\end{cases}
\end{align*}
\end{definition}

Applying Theorem \ref{thm:divisor} and the vanishing $R^jf_*\cO_Y(iK_Y)=0$ for $j\ge1$ to the exact sequence
\begin{align*}
0\to\cO_Y(iK_Y-F)\to\cO_Y(iK_Y)\to\cO_F(iK_Y|_F)\to0,
\end{align*}
we obtain
\begin{align}\label{eqn:delta}
\delta_P(i)&=\dim_kf_*\cO_Y(iK_Y)/f_*\cO_Y(iK_Y-F)\\
\nonumber&=h^0(\cO_F(iK_Y|_F))\\
\nonumber&=\chi(\cO_F(iK_Y|_F))\\
\nonumber&=\chi(\cO_Y(iK_Y))-\chi(\cO_Y(iK_Y-F)).
\end{align}

Let $I_0:=\{Q\ \textrm{with type $\frac{1}{r_Q}(1,-1,b_Q)$}\}$ be the basket of fictitious singularities from singularities on $Y$. Note that $b_Q$ is co-prime to $r_Q$. For $Q\in I_0$, let $f_Q$ denote the smallest non-negative integer such that $F\sim f_QK_Y$ at $Q$. By replacing $b_Q$ with $r_Q-b_Q$ if necessary, we may assume $v_Q:=\overline{f_Qb_Q}\le r_Q/2$. Set $I:=\{Q\in I_0\mid f_Q\neq0\}$.

With this notation, the singular Riemann--Roch formula computes the right-hand side of (\ref{eqn:delta}), to provide
\begin{align}\label{eqn:A}
\delta_P(i)=\frac{1}{6}F^3+\frac{1}{12}F\cdot c_2(Y)+\sum_{Q\in I}(A_Q(i)-A_Q(i-f_Q)),
\end{align}
where the contribution $A_Q(i)$ is given by
\begin{align*}
A_Q(i):=-\overline{i}\frac{r_Q^2-1}{12r_Q}+\sum_{j=1}^{\overline{i}-1}\frac{\overline{jb_Q}(r_Q-\overline{jb_Q})}{2r_Q}.
\end{align*}
The $A_Q(i)$ satisfies the formula
\begin{align*}
A_Q(i+1)-A_Q(i)&=-\frac{r_Q^2-1}{12r_Q}+B_Q(ib_Q)
\end{align*}
with
\begin{align*}
B_Q(i):=\frac{\overline{i}(r_Q-\overline{i})}{2r_Q}.
\end{align*}
Therefore by (\ref{eqn:A}), we have
\begin{align}\label{eqn:B}
\delta_P(i+1)-\delta_P(i)=\sum_{Q\in I}(B_Q(ib_Q)-B_Q(ib_Q-v_Q)).
\end{align}

\begin{lemma}\label{lem:index}
The $r_P$ equals the l.c.m.\ of $r_Q$ for all $Q\in I$.
\end{lemma}

\begin{proof}
Since $r_PK_Y=r_Pf^*K_X$ is a Cartier divisor about $f^{-1}(P)$, $r_Q$ divides $r_P$ for all $Q\in I$. On the other hand, we see that $r_P$ divides the l.c.m.\ of $r_Q$ by (\ref{eqn:B}) and the periodic properties of $\delta_P$, $B_Q$.
\end{proof}

\section{Boundedness of indices}\label{sec:boundedness}
We shall prove Theorem \ref{thm:main} in this section. Let $r_P$ denote the index of $X$ at $P$. We take a crepant blow-up $f\colon Y\to X$ with a divisor $F$ on $Y$ in Theorem \ref{thm:divisor}. We restrict the possibilities of $J:=\{(r_Q,v_Q)\}_{Q\in I}$ using (\ref{eqn:B}) for $i=0$.

\begin{lemma}\label{lem:J}
$J$ is one of the types in Table \ref{tbl:J}.
\begin{table}[ht]\caption{}\label{tbl:J}
\begin{tabular}{c|l|c}
\hline
\textup{type}&\multicolumn{1}{c|}{$J$} &$r_P$\\
\hline
\textup{1}   &$(2,1),(2,1),(2,1),(2,1)$&$2$  \\
\textup{2}   &$(2,1),(2,1),(4,2)$      &$4$  \\
\textup{3}   &$(2,1),(3,1),(6,1)$      &$6$  \\
\textup{4}   &$(2,1),(4,1),(4,1)$      &$4$  \\
\textup{5}   &$(3,1),(3,1),(3,1)$      &$3$  \\
\textup{6}   &$(4,2),(4,2)$            &$4$  \\
\textup{7}   &$(2,1),(6,3)$            &$6$  \\
\hline
\end{tabular}
\qquad
\begin{tabular}{c|l|c}
\hline
\textup{type}&\multicolumn{1}{c|}{$J$} &$r_P$\\
\hline
\textup{8}   &$(2,1),(8,2)$            &$8$  \\
\textup{9}   &$(3,1),(6,2)$            &$6$  \\
\textup{10}  &$(5,1),(5,2)$            &$5$  \\
\textup{11}  &$(8,4)$                  &$8$  \\
\textup{12}  &$(9,3)$                  &$9$  \\
\textup{13}  &$\emptyset$              &$1$  \\
\hline
\multicolumn{3}{c}{}                         \\
\end{tabular}
\end{table}
\end{lemma}

\begin{proof}
By Lemma \ref{lem:index}, $r_P$ is determined by $J$, and $r_P=1$ if and only if $J=\emptyset$. We assume $r_P>1$ from now on. Then (\ref{eqn:B}) for $i=0$ is written as
\begin{align}\label{eqn:B0}
\sum_{Q\in I}B_Q(v_Q)=1.
\end{align}
By the definition of $B_Q$ and $r_Q\ge2v_Q$, we have
\begin{align}\label{eqn:estimate}
v_Q/4\le B_Q(v_Q)<v_Q/2.
\end{align}
Then $J':=\{v_Q\}_{Q\in I}$, which satisfies (\ref{eqn:B0}) and (\ref{eqn:estimate}), should be one of
\begin{align*}
\{1,1,1,1\},\{1,1,2\},\{1,1,1\},\{2,2\},\{1,3\},\{1,2\},\{3\},\{4\}.
\end{align*}
For each of these candidates for $J'$, one can solve the equation (\ref{eqn:B0}) for $r_Q$ ($\ge2v_Q$) explicitly. Every solution is in Table \ref{tbl:J}. For example, suppose $J'=\{1,2\}$. We set $J=\{(r_1,1),(r_2,2)\}$. Then (\ref{eqn:B0}) becomes $1/r_1+4/r_2=1$. Thus $(r_1,r_2)=(2,8)$, $(3,6)$ or $(5,5)$, so $J$ is of type 8, 9, 10 respectively.
\end{proof}

By Lemma \ref{lem:J}, we have $r_P\le9$, and for Theorem \ref{thm:main} it is enough to exclude types 8, 11, 12. However, we derive a finer numerical classification by determining $\tilde{J}:=\{(r_Q,v_Q,b_Q)\}_{Q\in I}$.

\begin{theorem}
$\tilde{J}$ is one of the types in Table \ref{tbl:tildeJ}.
\begin{table}[ht]\caption{}\label{tbl:tildeJ}
\begin{tabular}{c|l|c}
\hline
\textup{type}&\multicolumn{1}{c|}{$\tilde{J}$}&$r_P$\\
\hline
\textup{1}   &$(2,1,1),(2,1,1),(2,1,1),(2,1,1)$&$2$ \\
\textup{3}   &$(2,1,1),(3,1,2),(6,1,5)$        &$6$ \\
\textup{4}   &$(2,1,1),(4,1,3),(4,1,3)$        &$4$ \\
\textup{5}   &$(3,1,2),(3,1,2),(3,1,2)$        &$3$ \\
\textup{10}  &$(5,1,4),(5,2,3)$                &$5$ \\
\textup{13}  &$\emptyset$                      &$1$ \\
\hline
\end{tabular}
\end{table}
\end{theorem}

\begin{proof}
By Lemma \ref{lem:J}, there exist only finitely many candidates for $\tilde{J}$. For each candidate, one can compute the right-hand side of (\ref{eqn:B}) explicitly. It must coincide with $\delta_P(i+1)-\delta_P(i)$, but such a coincidence happens only if $\tilde{J}$ is one of the types in Table \ref{tbl:tildeJ}.

Here we demonstrate for type 3. $\tilde{J}=\{(2,1,1),(3,1,b_2),(6,1,b_3)\}$ with $b_2=1$ or $2$ and $b_3=1$ or $5$. The (\ref{eqn:B}) for $i=1$ is $\delta_P(2)-\delta_P(1)=1,1/3,2/3,0$ when $(b_2,b_3)=(1,1),(1,5),(2,1),(2,5)$ respectively. Thus $(b_2,b_3)$ must be $(2,5)$, and in this case (\ref{eqn:B}) surely holds for any $i$.
\end{proof}

\begin{remark}
In simple cases, Theorem \ref{thm:main} is known by the classification.
\begin{enumerate}
\item
(Morrison \cite{Mo85}, Ishida--Iwashita \cite{II87})\
If $P$ is a cyclic quotient singularity, then $r_P=1$ except $\frac{1}{4n}(1,2n+1,-2)$ ($n\ge2$), $\frac{1}{14}(1,9,11)$, $\frac{1}{9}(1,4,7)$, with $r_P=2,2,3$ respectively.
\item
(Hayakawa--Takeuchi \cite{HT87})\
If $P$ is an isolated singularity which is a cyclic quotient of a hypersurface singularity, then $r_P\le4$. The only case when $r_P=4$ is $o\in(x_1x_2+x_3^2+x_4^2=0)\subset\bA_{x_1x_2x_3x_4}^4/\bZ_{8}(1,5,3,7)$.
\end{enumerate}
\end{remark}

\section{Minimal discrepancies}\label{sec:md}
To begin with, we provide an example which explains the need of $P$ being a crepant centre in Theorem \ref{thm:main} even for a strictly canonical singularity. A similar example exists also for a $3$-fold strictly log canonical singularity \cite[Example 6.1]{F01}.

\begin{example}\label{exl:unbounded}
Let $r\in\bN$. Let $P\in X$ be the germ
\begin{align*}
o\in(x_1x_2+x_3^2=0)\subset\bA_{x_1x_2x_3x_4}^4/\bZ_r(1,-1,0,1),
\end{align*}
which is singular along the $x_4$-axis $C$. Let $f\colon Y\to X$ be the weighted blow-up with weights $\wt(x_1,x_2,x_3,x_4)=\frac{1}{r}(1,r-1,r,1)$. Then $K_Y=f^*K_X+\frac{1}{r}E$ with the exceptional divisor $E$, and $Y$ has $2$ terminal quotient singularities of types $\frac{1}{r-1}(1,-1,1)$ and $\frac{1}{r}(1,-1,1)$ outside the strict transform $C_Y$ of $C$. Let $g\colon Z\to Y$ be the blow-up with centre $C_Y$. The $g$ is a crepant blow-up and $Z$ is smooth about $g^{-1}(C_Y)$. Hence $X$ has canonical singularities with a crepant centre $C$, but $P$ is not a crepant centre. The index of $X$ at $P$ is $r$.
\end{example}

We focus on the minimal discrepancy to grasp this phenomenon. For a normal $\bQ$-Gorenstein singularity $P\in X$, the \textit{minimal discrepancy} $\md_PX$ of $X$ at $P$ is the infimum of discrepancies $a_E(X)$ for all divisors $E$ over $X$ with $c_X(E)=P$. Note that $\md_PX\in\{-\infty\}\cup[-1,\infty)$, and $P\in X$ is log canonical if and only if $\md_PX\ge-1$.

In Example \ref{exl:unbounded}, we have $\md_PX=1/r$. Shokurov formulated a question on the boundedness of indices in terms of minimal discrepancies.

\begin{question}[Shokurov]\label{qsn:Shokurov}
For each $(n,a)\in\bN\times[-1,\infty)$, does there exist a number $r(n,a)$ such that the index of an arbitrary $n$-fold log canonical singularity $P\in X$ with $\md_PX=a$ is at most $r(n,a)$?
\end{question}

He raised its weaker variant for canonical singularities.

\begin{question'}
For each $(n,a)\in\bN\times[0,\infty)$, does there exist a number $r'(n,a)$ such that the index of an arbitrary $n$-fold canonical singularity $P\in X$ with $\md_PX=a$ is at most $r'(n,a)$?
\end{question'}

The result of Ishii and Fujino gives $r(3,-1)=66$ for Question \ref{qsn:Shokurov}. Theorem \ref{thm:main} gives $r'(3,0)=6$ for Question \ref{qsn:Shokurov}$'$. Further, we provide an affirmative answer to Question \ref{qsn:Shokurov}$'$ for $n=3$.

\begin{theorem}\label{thm:md}
Question \ref{qsn:Shokurov}$'$ is true for $n=3$. More precisely, the minimal discrepancy of a $3$-fold canonical singularity is $0$, $1/r$ \textup{(}$r\in\bN$\textup{)} or $2$, and one can take
\begin{align*}
r'(3,0)=6,\qquad r'(3,1/r)=r!,\qquad r'(3,2)=1.
\end{align*}
\end{theorem}

\begin{proof}
Let $P\in X$ be a $3$-fold canonical singularity with index $r_P$. We shall verify the statement for any such $P$. We take a crepant blow-up $f\colon Y\to X$ with $Y$ terminal by Corollary \ref{cor:crepant}.

Suppose $\dim f^{-1}(P)=0$, that is, $P$ is terminal. Then it suffices to recall $\md_PX=1/r_P$ \cite{Km92}, \cite{Ma96} for terminal $P$ except for smooth $P$.

Suppose $\dim f^{-1}(P)=1$. For any curve $C\subset f^{-1}(P)$, the blow-up of $Y$ with centre $C$ generates a divisor $E$ with $a_E(X)=1$. Together with the mentioned result \cite{Km92}, \cite{Ma96}, we see that $\md_PX$ is the minimum of $1/r_Q$ for all $Q\in f^{-1}(P)$, where $r_Q$ denotes the index of $Y$ at $Q$. Hence $\md_PX=1/r$ with $r\in\bN$ and $r_Q\le r$ for all $Q\in f^{-1}(P)$. Thus $r!K_Y$ is a Cartier divisor about $f^{-1}(P)$, so $r_P\mid r!$ by \cite[Corollary 1.5]{Km88}.

Suppose $\dim f^{-1}(P)=2$. Then $P$ is a crepant centre, that is, $\md_PX=0$. The statement holds by Theorem \ref{thm:main}.
\end{proof}

{\small
\begin{acknowledgements}
I was asked the question on the boundedness of indices by Professor V. V. Shokurov at Johns Hopkins University in 2005. This question was raised again in the workshop at American Institute of Mathematics in 2012. I am grateful to him for the discussions. The research was partially supported by Grant-in-Aid for Young Scientists (A) 24684003.
\end{acknowledgements}
}


\begin{thebibliography}{99}
\bibitem{BCHM10}
C. Birkar, P. Cascini, C. D. Hacon and J. McKernan,
Existence of minimal models for varieties of log general type,
J. Am.\ Math.\ Soc.\ \textbf{23}, No.\ 2, 405-468 (2010)
\bibitem{F01}
O. Fujino,
The indices of log canonical singularities,
Am.\ J. Math.\ \textbf{123}, No.\ 2, 229-253 (2001)
\bibitem{HT87}
T. Hayakawa and K. Takeuchi,
On canonical singularities of dimension three,
Jap.\ J. Math., New Ser.\ \textbf{13}, No.\ 1, 1-46 (1987)
\bibitem{II87}
M.-N. Ishida and N. Iwashita,
Canonical cyclic quotient singularities of dimension three,
\textit{Complex analytic singularities}, Tsukuba 1984, Adv.\ Stud.\ Pure Math.\ \textbf{8}, 135-151 (1987)
\bibitem{I00}
S. Ishii,
The quotients of log-canonical singularities by finite groups,
\textit{Singularities}, Sapporo 1998, Adv.\ Stud.\ Pure Math.\ \textbf{29}, 135-161 (2000)
\bibitem{K01}
M. Kawakita,
Divisorial contractions in dimension three which contract divisors to smooth points,
Invent.\ Math.\ \textbf{145}, No.\ 1, 105-119 (2001)
\bibitem{K05}
M. Kawakita,
Three-fold divisorial contractions to singularities of higher indices,
Duke Math.\ J. \textbf{130}, No.\ 1, 57-126 (2005)
\bibitem{Km88}
Y. Kawamata,
Crepant blowing-up of $3$-dimensional canonical singularities and its application to degenerations of surfaces,
Ann.\ Math.\ (2) \textbf{127}, No.\ 1, 93-163 (1988)
\bibitem{Km92}
Y. Kawamata,
The minimal discrepancy coefficients of terminal singularities in dimension $3$, Appendix to V. V. Shokurov, Three-dimensional log perestroikas,
Izv.\ Ross.\ Akad.\ Nauk Ser.\ Mat.\ \textbf{56}, No.\ 1, 105-201, Appendix 201-203 (1992)
\bibitem{KMM87}
Y. Kawamata, K. Matsuda and K.Matsuki,
Introduction to the minimal model problem,
\textit{Algebraic geometry}, Sendai 1985, Adv.\ Stud.\ Pure Math.\ \textbf{10}, 283-360 (1987)
\bibitem{K+92}
J.\ Koll\'ar (ed.),
\textit{Flips and abundance for algebraic threefolds},
Ast\'erisque \textbf{211} (1992)
\bibitem{Ma96}
D. Markushevich,
Minimal discrepancy for a terminal cDV singularity is $1$,
J. Math.\ Sci.\ Tokyo \textbf{3}, No.\ 2, 445-456 (1996)
\bibitem{Mo85}
D. R. Morrison,
Canonical quotient singularities in dimension three,
Proc.\ Am.\ Math.\ Soc.\ \textbf{93}, No.\ 3, 393-396 (1985)
\bibitem{R87}
M. Reid,
Young person's guide to canonical singularities,
\textit{Algebraic geometry}, Bowdoin 1985, Proc.\ Symp.\ Pure Math.\ \textbf{46}, Part 1, 345-414 (1987)
\end{thebibliography}
\end{document}